\documentclass[11pt]{article}
\pdfoutput=1
\usepackage{amsmath,amssymb,amsfonts}
\usepackage{epsfig,graphics,subfig}
\usepackage[USenglish]{babel}
\usepackage{amsthm}

\theoremstyle{plain}
\newtheorem{theorem}{Theorem}
\newtheorem{lemma}{Lemma}

\newtheorem{corollary}{Corollary}

\def\R{\mathbb{R}}

\newcommand{\e}{\mathrm{e}}

\newcommand{\var}{\mathrm{var}}

\newcommand{\pdim}{\ensuremath{p}}
\newcommand{\kdim}{\ensuremath{k}}
\newcommand{\numobs}{\ensuremath{n}}
\newcommand{\spgam}{\ensuremath{\gamma}}

\newcommand{\Numset}{\ensuremath{N}}

\newcommand{\defn}{\ensuremath{: \, =}}

\newcommand{\Sset}{\ensuremath{S}}
\newcommand{\Uset}{\ensuremath{U}}

\newcommand{\Shat}{\ensuremath{\widehat{\Sset}}}
\newcommand{\betahat}{\ensuremath{\widehat{\beta}}}

\newcommand{\mprob}{\ensuremath{\mathbb{P}}}
\newcommand{\Exs}{\ensuremath{\mathbb{E}}}

\newcommand{\betamin}{\ensuremath{\beta_{min}}}
\newcommand{\minval}{\betamin}

\long\def\comment#1{}

\newcommand{\myparagraph}[1]{\noindent {\bf{#1}}}

\newcommand{\real}{\ensuremath{\mathbb{R}}}

\newcommand{\Bin}{\ensuremath{\operatorname{Bin}}}
\newcommand{\Ber}{\ensuremath{\operatorname{Ber}}}

\newcommand{\MYLOGNUM}{\ensuremath{\log {\pdim \choose \kdim}-1}}

\newcommand{\myboundone}{\ensuremath{f_1(\pdim, \kdim, \minval)}}
\newcommand{\myboundtwo}{\ensuremath{f_2(\pdim, \kdim, \minval)}}

\newcommand{\spboundone}{\ensuremath{g_1(\pdim, \kdim, \minval, \spgam)}}
\newcommand{\spboundtwo}{\ensuremath{g_2(\pdim, \kdim, \minval, \spgam)}}

\newcommand{\pastSNR}{\ensuremath{\operatorname{SNR}}}

\newcommand{\Ytil}{\ensuremath{\widetilde{Y}}}
\newcommand{\betatil}{\ensuremath{\widetilde{\beta}}}

\newcommand{\Ybar}{\ensuremath{Y}}
\newcommand{\betabar}{\ensuremath{\bar{\beta}}}

\newcommand{\myprob}[1]{\ensuremath{\frac{#1}{{\pdim \choose \kdim}}}}

\newcommand{\Hbin}{\ensuremath{H_{binary}}}

\newcommand{\jstar}{\ensuremath{j^\star}}

\newcommand{\SigClass}{\ensuremath{\mathcal{C}}}

\newcommand{\psibar}{\ensuremath{\overline{\psi}}}

\makeatletter
\def\@cite#1#2{[\if@tempswa #2 \fi #1]}
\makeatother

\makeatletter
\long\def\@makecaption#1#2{
        \vskip 0.8ex
        \setbox\@tempboxa\hbox{\small {\bf #1:} #2}
        \parindent 1.5em  
        \dimen0=\hsize
        \advance\dimen0 by -3em
        \ifdim \wd\@tempboxa >\dimen0
                \hbox to \hsize{
                        \parindent 0em
                        \hfil 
                        \parbox{\dimen0}{\def\baselinestretch{0.96}\small
                                {\bf #1.} #2
                                } 
                        \hfil}
        \else \hbox to \hsize{\hfil \box\@tempboxa \hfil}
        \fi
        }
\makeatother

\begin{document}

\begin{center}
{\bf{\LARGE{Information-theoretic limits on sparse signal recovery:\\
Dense versus sparse measurement matrices}}}

\vspace*{.3in}

{\large{
\begin{tabular}{ccccc}
Wei Wang$^\star$ & & Martin J. Wainwright$^{\dagger,\star}$ & &
Kannan Ramchandran$^\star$
\end{tabular}
\begin{tabular}{c}
{\texttt{$\{$wangwei, wainwrig, kannanr$\}$@eecs.berkeley.edu}}
\end{tabular}
}}

\vspace*{.2in}

\begin{tabular}{c}
Department of Electrical Engineering and Computer Sciences$^\star$, and \\
Department of Statistics$^\dagger$ \\
University of California, Berkeley
\end{tabular}

\vspace*{.2in}

\begin{tabular}{c}
Technical Report \\
Department of Statistics,  UC Berkeley \\
May 2008
\end{tabular}

\vspace*{.2in}

\begin{abstract}
We study the information-theoretic limits of exactly recovering the
support of a sparse signal using noisy projections defined by various
classes of measurement matrices.  Our analysis is high-dimensional in
nature, in which the number of observations $\numobs$, the ambient
signal dimension $\pdim$, and the signal sparsity $\kdim$ are all
allowed to tend to infinity in a general manner. This paper makes two
novel contributions.  First, we provide sharper necessary conditions
for exact support recovery using general (non-Gaussian) dense
measurement matrices.  Combined with previously known sufficient
conditions, this result yields sharp characterizations of when the
optimal decoder can recover a signal for various scalings of the
sparsity $\kdim$ and sample size $\numobs$, including the important
special case of linear sparsity ($\kdim = \Theta(\pdim)$) using a
linear scaling of observations ($\numobs = \Theta(\pdim)$).  Our
second contribution is to prove necessary conditions on the number of
observations $\numobs$ required for asymptotically reliable recovery
using a class of $\spgam$-sparsified measurement matrices, where the
measurement sparsity $\spgam(\numobs, \pdim, \kdim) \in (0,1]$
corresponds to the fraction of non-zero entries per row.  Our analysis
allows general scaling of the quadruplet $(\numobs, \pdim, \kdim,
\spgam)$, and reveals three different regimes, corresponding to
whether measurement sparsity has no effect, a minor effect, or a
dramatic effect on the information-theoretic limits of the subset
recovery problem.
\end{abstract}

\end{center}

\myparagraph{Keywords:} Sparsity recovery; sparse random matrices;
subset selection; compressive sensing; signal denoising; sparse
approximation; information-theoretic bounds; Fano's inequality.

%%%%%%%%%%%%%%%%%%%%%%%%%%%%%%%%%%%%%%%%%%%%%%%%%%%%%%%%%%%%%%%%%%%%%%%%%%%%%

\section{Introduction}
\label{SecIntro}

The problem of estimating a $\kdim$-sparse vector $\beta \in \R^\pdim$
based on a set of $\numobs$ noisy linear observations is of broad
interest, arising in subset selection in regression, graphical model
selection, group testing, signal denoising, sparse approximation, and
compressive sensing.  A large body of recent work
(e.g.,~\cite{Chen98,Donoho06,DonElTem06,CandesTao05,CanRomTao06,Meinshausen06,Tibshirani96,Tropp06,Wainwright06a,
Gil06,ComMut05,WaGaRa07, XuHa07,SaBaBa06a}) has analyzed the use of
$\ell_1$-relaxation methods for estimating high-dimensional sparse
signals, and established conditions (on signal sparsity and the choice
of measurement matrices) under which they succeed with high
probability.

Of complementary interest are the information-theoretic limits of the
sparsity recovery problem, which apply to the performance of any
procedure regardless of its computational complexity.  Such analysis
has two purposes: first, to demonstrate where known polynomial-time
methods achieve the information-theoretic bounds, and second, to
reveal situations in which current methods are sub-optimal.  An
interesting question which arises in this context is the effect of the
choice of measurement matrix on the information-theoretic limits of
sparsity recovery.  As we will see, the standard Gaussian measurement
ensemble is an optimal choice in terms of minimizing the number of
observations required for recovery.  However, this choice produces
highly dense measurement matrices, which may lead to prohibitively
high computational complexity and storage requirements.  Sparse
matrices can reduce this complexity, and also lower communication cost
and latency in distributed network and streaming applications.  On the
other hand, such measurement sparsity, though beneficial from the
computational standpoint, may reduce statistical efficiency by
requiring more observations to decode.  Therefore, an important issue
is to characterize the trade-off between measurement sparsity and
statistical efficiency.

With this motivation, this paper makes two contributions.  First, we
derive sharper necessary conditions for exact support recovery,
applicable to a general class of dense measurement matrices (including
non-Gaussian ensembles).  In conjunction with the sufficient
conditions from previous work~\cite{Wainwright06_info}, this analysis
provides a sharp characterization of necessary and sufficient
conditions for various sparsity regimes.  Our second contribution is
to address the effect of measurement sparsity, meaning the fraction
$\spgam \in (0,1]$ of non-zeros per row in the matrices used to
collect measurements.  We derive lower bounds on the number of
observations required for exact sparsity recovery, as a function of
the signal dimension $\pdim$, signal sparsity $\kdim$, and measurement
sparsity $\spgam$.  This analysis highlights a trade-off between the
statistical efficiency of a measurement ensemble and the computational
complexity associated with storing and manipulating it.

The remainder of the paper is organized as follows.  We first define
our problem formulation in Section~\ref{SecProblem}, and then discuss
our contributions and some connections to related work in
Section~\ref{SecRelated}.  Section~\ref{SecMain} provides precise
statements of our main results, as well as a discussion of their
consequences.  Section~\ref{SecProof} provides proofs of the necessary
conditions for various classes of measurement matrices, while proofs
of more technical lemmas are given in the appendices.  Finally, we
conclude and discuss open problems in Section~\ref{SecDiscuss}.

\subsection{Problem formulation} 
\label{SecProblem}

There are a variety of problem formulations in the growing body of
work on compressive sensing and related areas.  The signal model may
be exactly sparse, approximately sparse, or compressible (i.e. that
the signal is approximately sparse in some orthonormal basis).  The
most common signal model is a deterministic one, although Bayesian
formulations are also possible.  In addition, the observation model
can be either noiseless or noisy, and the measurement matrix can be
random or deterministic.  Furthermore, the signal recovery can be
perfect or approximate, assessed by various error metrics (e.g.,
$\ell_q$-norms, prediction error, subset recovery).

In this paper, we consider a deterministic signal model, in which
$\beta \in \real^\pdim$ is a fixed but unknown vector with exactly
$\kdim$ non-zero entries.  We refer to $\kdim$ as the \emph{signal
sparsity} and $\pdim$ as the \emph{signal dimension}, and define the
support set of $\beta$ as
\begin{eqnarray}\
\Sset & \defn & \{i \in \{1,\ldots,p\} \, \mid \, \beta_i \neq 0\}.
\end{eqnarray}  
Note that there are $\Numset = {\pdim \choose \kdim}$ possible support
sets, corresponding to the $\Numset$ possible $\kdim$-dimensional
subspaces in which $\beta$ can lie.  We are given a vector of
$\numobs$ noisy observations $Y \in \real^\numobs$, of the form
\begin{eqnarray}
\label{EqnNoisyObs}
   Y &=& X \beta + W,
\end{eqnarray}
where $X \in \R^{n \times p}$ is the measurement matrix, and \mbox{$W
\sim N(0,\sigma^2 I_{n \times n})$} is additive Gaussian noise.  Our
results apply to various classes of dense and $\spgam$-sparsified
measurement matrices, which will be defined concretely in
Section~\ref{SecMain}.  Throughout this paper, we assume without loss
of generality that $\sigma^2=1$, since any scaling of $\sigma$ can be
accounted for in the scaling of $\beta$.

Our goal is to perform exact recovery of the support set $\Sset$,
which corresponds to a standard model selection error criterion.  More
precisely, we measure the error between the estimate $\betahat$ and
the true signal $\beta$ using the $\{0,1\}$-valued loss function:
\begin{equation}
\label{EqnError}
\rho(\betahat,\beta) \defn \mathbb{I} \left[\{\betahat_i \neq 0, \;
   \forall i \in S\} \cap \{\betahat_j = 0, \; \forall j \not\in
   S\}\right].
\end{equation}
The results of this paper apply to arbitrary decoders.  Any decoder is
a mapping $g$ from the observations $Y$ to an estimated subset
$\Shat=g(Y)$.  Let $\mprob[g(Y) \neq \Sset \, \mid \, \Sset]$ be the
conditional probability of error given that the true support is $S$.
Assuming that $\beta$ has support $\Sset$ chosen uniformly at random
over the $\Numset$ possible subsets of size $\kdim$, the average
probability of error is given by
\begin{eqnarray}
   p_{err} &=& \frac{1}{{p \choose k}} \sum_S \mprob[g(Y) \neq \Sset
   \, \mid \, \Sset].
\end{eqnarray}
We say that sparsity recovery is asymptotically reliable if $p_{err}
\rightarrow 0$ as $\numobs \rightarrow \infty$.  Since we are trying
to recover the support exactly from noisy measurements, our results
necessarily involve the minimum value of $\beta$ on its support,
\begin{eqnarray}
\label{EqnDefnBetaMin} 
\betamin & \defn & \min_{i \in \Sset} |\beta_i|.
\end{eqnarray}
In particular, our results apply to decoders that operate over the signal
class
\begin{eqnarray}
\label{EqnDefnSigClass}
\SigClass(\betamin) & \defn & \{\beta \in \real^p \, \mid \, |\beta_i|
\geq \betamin \; \forall i \in \Sset\}.
\end{eqnarray}

With this set-up, our goal is to find necessary conditions on the
parameters $(\numobs, \pdim, \kdim, \betamin, \spgam)$ that any
decoder, regardless of its computational complexity, must satisfy for
asymptotically reliable recovery to be possible.  We are interested in
lower bounds on the number of measurements $\numobs$, in general
settings where both the signal sparsity $\kdim$ and the measurement
sparsity $\spgam$ are allowed to scale with the signal dimension
$\pdim$.  As our analysis shows, the appropriate notion of rate for
this problem is $R = \frac{\log {\pdim \choose \kdim}}{\numobs}$.

\subsection{Our contributions}
\label{SecRelated}

One body of past work~\cite{Fletcher06,Bar06,AerZhaVen07} has focused
on the information-theoretic limits of sparse estimation under
$\ell_2$ and other distortion metrics, using power-based SNR measures
of the form
\begin{eqnarray}
\label{EqnDefnPastSNR}
 \pastSNR &\defn& \frac{\Exs[\|X\beta\|_2^2]}{\Exs[\|W\|_2^2]} \;\; =
   \;\; \|\beta\|_2^2.
\end{eqnarray}
(Note that the second equality assumes that the noise variance
$\sigma^2 = 1$, and that the measurement matrix is standardized, with
each element $X_{ij}$ having zero-mean and variance one.)  It is
important to note that the power-based SNR~\eqref{EqnDefnPastSNR},
though appropriate for $\ell_2$-distortion, is not suitable for the
support recovery problem.  Although the minimum value is related to
this power-based measure by the inequality $\kdim \minval^2 \leq
\pastSNR$, for the ensemble of signals $\SigClass(\betamin)$ defined
in equation~\eqref{EqnDefnSigClass}, the $\ell_2$-based SNR
measure~\eqref{EqnDefnPastSNR} can be made arbitrarily large, while
still having one coefficient $\beta_i$ equal to the minimum value
(assuming that $\kdim > 1$).  Consequently, as our results show, it is
possible to generate problem instances for which support recovery is
arbitrarily difficult---in particular, by sending $\betamin
\rightarrow 0$ at an arbitrarily rapid rate---even as the power-based
SNR~\eqref{EqnDefnPastSNR} becomes arbitrarily large.

The paper~\cite{Wainwright06_info} was the first to consider the
information-theoretic limits of exact subset recovery using dense
Gaussian measurement ensembles, explicitly identifying the minimum
value $\betamin$ as the key parameter.  This analysis yielded
necessary and sufficient conditions on general quadruples $(\numobs,
\pdim, \kdim, \betamin)$ for asymptotically reliable recovery.
Subsequent work~\cite{Reeves, Harvard} has extended this type of
analysis to the criterion of partial support recovery.  In this paper,
we consider only exact support recovery, but provide results for
general dense measurement ensembles, thereby extending previous
results.  In conjunction with known sufficient
conditions~\cite{Wainwright06_info}, one consequence of our first main
result (Theorem~\ref{ThmDenseNew}, below) is a set of sharp necessary
and sufficient conditions for the optimal decoder to recover the
support of a signal with linear sparsity ($\kdim = \Theta(\pdim)$),
using only a linear fraction of observations ($\numobs =
\Theta(\pdim)$).  Moreover, for the special case of the standard
Gaussian ensemble, Theorem~\ref{ThmDenseNew} also recovers some
results independently obtained in concurrent work by
Reeves~\cite{Reeves}, and Fletcher et al.~\cite{Fletcher08}.

We then consider the effect of measurement sparsity, which we assess
in terms of the fraction $\spgam \in (0,1]$ of non-zeros per row of
the the measurement matrix $X$.  Some past work in compressive sensing
has proposed computationally efficient recovery methods based on
sparse measurement matrices, including work inspired by expander
graphs and coding theory~\cite{XuHa07,SaBaBa06a}, sparse random
projections for Johnson-Lindenstrauss embeddings~\cite{WaGaRa07}, and
sketching and group testing~\cite{Gil06,ComMut05}.  All of this work
deals with the noiseless observation model, in contrast to the noisy
observation model~\eqref{EqnNoisyObs} considered here.  The
paper~\cite{AerZhaVen07} provides results on sparse measurements for
noisy problems and distortion-type error metrics, using a Bayesian
signal model and power-based SNR that is not appropriate for the
subset recovery problem.  Also, some concurrent work~\cite{OmiWai08}
provides sufficient conditions for support recovery using the Lasso
($\ell_1$-constrained quadratic programming) for appropriately
sparsified ensembles.  These results can be viewed as complementary to
the information-theoretic analysis of this paper.  In this paper, we
characterize the inherent trade-off between measurement sparsity and
statistical efficiency.  More specifically, our second main result
(Theorem~\ref{ThmSparseNew}, below) provides necessary conditions for
exact support recovery, using $\spgam$-sparsified Gaussian measurement
matrices (see equation~\eqref{EqnGamSparse}), for general scalings of
the parameters $(\numobs, \pdim, \kdim, \betamin, \spgam)$.  This
analysis reveals three regimes of interest, corresponding to whether
measurement sparsity has no effect, a small effect, or a significant
effect on the number of measurements necessary for recovery.  Thus,
there exist regimes in which measurement sparsity fundamentally alters
the ability of any method to decode.

\section{Main results and consequences}
\label{SecMain}

In this section, we state our main results, and discuss some of their
consequences.  Our analysis applies to random ensembles of measurement
matrices $X \in \real^{\numobs \times \pdim}$, where each entry
$X_{ij}$ is drawn i.i.d. from some underlying distribution.  The most
commonly studied random ensemble is the standard Gaussian case, in
which each $X_{ij} \sim N(0,1)$.  Note that this choice generates a
highly dense measurement matrix $X$, with $\numobs \pdim$ non-zero
entries.  Our first result (Theorem~\ref{ThmDenseNew}) applies to more
general ensembles that satisfy the moment conditions $\Exs[X_{ij}] =
0$ and $\var(X_{ij}) = 1$, which allows for a variety of non-Gaussian
distributions (e.g., uniform, Bernoulli etc.).  In addition, we also
derive results (Theorem~\ref{ThmSparseNew}) for $\spgam$-sparsified
matrices $X$, in which each entry $X_{ij}$ is i.i.d.  drawn according
to
\begin{eqnarray}
\label{EqnGamSparse}
   X_{ij} = \left\{
   \begin{array}{rl}
       N(0,\frac{1}{\spgam}) & \text{ w.p. } \spgam\\ 
       0 & \text{ w.p. } 1 - \spgam
   \end{array}.
   \right.
\end{eqnarray}
Note that when $\spgam=1$, $X$ is exactly the standard Gaussian
ensemble.  We refer to the sparsification parameter $0 \leq \spgam\leq
1$ as the \emph{measurement sparsity}.  Our analysis allows this
parameter to vary as a function of $(\numobs, \pdim, \kdim)$.

\subsection{Tighter bounds on dense ensembles} 

We begin by noting an analogy to the Gaussian channel coding problem
that yields a straightforward but loose set of necessary conditions.
Support recovery can be viewed as a channel coding problem, in which
there are $N = {\pdim \choose \kdim}$ possible support sets of
$\beta$, corresponding to messages to be sent over a Gaussian channel
with noise variance $1$.  The effective code rate is then $R =
\frac{\log{\pdim \choose \kdim}}{\numobs}$.  If each support set
$\Sset$ is encoded as the codeword $c(\Sset) = X \beta$, where $X$ has
i.i.d. Gaussian entries, then by standard Gaussian channel capacity
results, we immediately obtain a lower bound on the number of
observations $\numobs$ necessary for asymptotically reliable recovery,
\begin{eqnarray}
\label{EqnGaussChannel}
   \numobs & > & \frac{\log{\pdim \choose \kdim}}{\frac{1}{2}
   \log\left(1+ \|\beta\|_2^2\right)}.
\end{eqnarray} 
This bound is tight for $\kdim = 1$ and Gaussian measurements, but
loose in general. As Theorem~\ref{ThmDenseNew} clarifies, there are
additional elements in the support recovery problem that distinguish
it from a standard Gaussian coding problem: first, the signal power
$\|\beta\|_2^2$ does not capture the inherent problem difficulty for
$\kdim > 1$, and second, there is overlap between support sets for
$\kdim > 1$.  The following result provides sharper conditions on
subset recovery.
\begin{theorem}[General ensembles]
\label{ThmDenseNew}
Let the measurement matrix \mbox{$X \in \real^{\numobs \times \pdim}$}
be drawn with i.i.d. elements from any distribution with zero-mean and
variance one.  Then a necessary condition for asymptotically reliable
recovery over the signal class $\SigClass(\betamin)$ is
\begin{eqnarray}
\label{EqnNewDenseConverse}
   \numobs & > & \max \big \{ \myboundone, \quad \myboundtwo, \quad
   \kdim-1 \big \},
\end{eqnarray}
where 
\begin{subequations}
\begin{eqnarray}
\myboundone & \defn & \frac{\MYLOGNUM}{\frac{1}{2}\log\left(1 + \kdim
\minval^2 (1 - \frac{\kdim}{\pdim })\right)} \\
\myboundtwo & \defn & \frac{\log (\pdim - \kdim +1) - 1}{\frac{1}{2}
\log\left(1+ \minval^2 (1-\frac{1}{\pdim-\kdim+1})\right)}.
\end{eqnarray}
\end{subequations}
\end{theorem}

The proof of Theorem~\ref{ThmDenseNew}, given in
Section~\ref{SecProof}, uses Fano's inequality to bound the
probability of error of any recovery method.  In addition to the
standard Gaussian ensemble ($X_{ij} \sim N(0,1)$), this result also
covers matrices from other common ensembles (e.g., Bernoulli $X_{ij}
\in \, \{-1, +1\}$).  It generalizes and strengthens earlier results
on subset recovery~\cite{Wainwright06_info}.  Note that $\|\beta\|_2^2
\geq \kdim \minval^2$ (with equality in the case when
$|\beta_i|=\minval$ for all indices $i \in \Sset$), so that this bound
is strictly tighter than the intuitive bound~\eqref{EqnGaussChannel}.
Moreover, by fixing the value of $\beta$ at $(\kdim-1)$ indices to
$\betamin$ and allowing the last component of $\beta$ to tend to
infinity, we can drive the power $\|\beta\|_2^2$ to infinity, while
still having the minimum enter the lower bound.

The necessary conditions in Theorem~\ref{ThmDenseNew} can be compared
against the sufficient conditions in
Wainwright~\cite{Wainwright06_info} for exact support recovery using
the standard Gaussian ensemble, as shown in Table~\ref{TabDense}.  We
obtain tight necessary and sufficient conditions in the regime of
linear signal sparsity (meaning $\kdim/\pdim = \alpha$ for some
$\alpha \in (0,1)$), under various scalings of the minimum value
$\betamin$.  We also obtain tight matching conditions in the regime of
sublinear signal sparsity (in which $\kdim/\pdim \rightarrow 0$), when
$\kdim\betamin^2 = \Theta(1)$.  There remains a slight gap, however,
in the sublinear sparsity regime when $\kdim\betamin^2 \rightarrow
\infty$ (see bottom two rows in Table~\ref{TabDense}).
Moreover, these information-theoretic bounds can be compared to the
recovery threshold of $\ell_1$-constrained quadratic programming,
known as the Lasso~\cite{Wainwright06a}.  This comparison reveals that
whenever $\kdim\betamin^2 = \Theta(1)$ (in both the linear and
sublinear sparsity regimes), then $\Theta(\kdim \log(\pdim-\kdim))$
observations are necessary and sufficient for sparsity recovery, and
hence the Lasso method is information-theoretically optimal.  In
contrast, when $\kdim\betamin^2 \rightarrow \infty$ and $\kdim/\pdim =
\alpha$, there is a gap between the performance of the Lasso and the
information-theoretic bounds.

\begin{table} 
\centering
\begin{tabular}{|c|c|c|}
   \hline
   & Necessary conditions & Sufficient conditions \\
   & (Theorem~\ref{ThmDenseNew}) & (Wainwright~\cite{Wainwright06_info}) \\
   \hline\hline
   $\begin{array}{c}
   \kdim = \Theta(\pdim) \\ \minval^2 = \Theta(\frac{1}{\kdim})
   \end{array}$ 
   & $\Theta(\pdim \log \pdim)$ & $\Theta(\pdim \log \pdim)$ \\
   \hline
   $\begin{array}{c}
   \kdim = \Theta(\pdim) \\ \minval^2 = \Theta(\frac{\log\kdim}{\kdim})
   \end{array}$
   & $\Theta(\pdim)$ & $\Theta(\pdim)$ \\ 
   \hline
   $\begin{array}{c}
   \kdim = \Theta(\pdim) \\ \minval^2 = \Theta(1)
   \end{array}$
   & $\Theta(\pdim)$ & $\Theta(\pdim)$ \\
   \hline
   $\begin{array}{c}
   \kdim = o(\pdim) \\ \minval^2 = \Theta(\frac{1}{\kdim})
   \end{array}$ 
   & $\Theta(\kdim \log(\pdim-\kdim))$ & $\Theta(\kdim \log(\pdim-\kdim))$ \\
   \hline \hline
%%%%%%%%%%%%%%%%%%%%%%%%%%%%%%%%%%%%%%%%%%%%%%%%%%%%%%%%%%%%%%%%%%%%%%%%%%%%%%%
   $\begin{array}{c}
   \kdim = o(\pdim) \\ \minval^2 = \Theta(\frac{\log\kdim}{\kdim})
   \end{array}$
   & $\max \left\{\Theta\left( \frac{\kdim \log
   \frac{\pdim}{\kdim}}{\log\log\kdim} \right),
   \Theta\left(\frac{\kdim \log(\pdim-\kdim)}{\log\kdim} \right)
   \right\}$
   & $\Theta\left(\kdim \log\frac{\pdim}{\kdim}\right)$ \\
   \hline
   $\begin{array}{c}
   \kdim = o(\pdim) \\ \minval^2 = \Theta(1)
   \end{array}$
   & $\max \left\{ \Theta\left(\frac{\kdim \log
   \frac{\pdim}{\kdim}}{\log\kdim}\right), \Theta(\kdim) \right\}$
   & $\Theta\left(\kdim \log\frac{\pdim}{\kdim}\right)$ \\
   \hline
\end{tabular}
\caption{Tight necessary and sufficient conditions on the number of
observations $\numobs$ required for exact support recovery are
obtained in several regimes of interest.}
\label{TabDense} 
\end{table}

Theorem~\ref{ThmDenseNew} has some consequences related to results
proved in concurrent work.  Reeves and Gastpar~\cite{Reeves} have
shown that in the regime of linear sparsity $\kdim/\pdim = \alpha >
0$, if any decoder is given only a linear fraction sample size
(meaning that $\numobs = \Theta(\pdim)$), then in order to recover the
support exactly, one must have $\kdim \minval^2 \rightarrow +\infty$.
This result is one corollary of Theorem~\ref{ThmDenseNew}, since if
$\minval^2 = \Theta(1/\kdim)$, then we have
\begin{eqnarray*}
\numobs \; > \; \frac{ \log (\pdim - \kdim + 1)-1}{\frac{1}{2}\log(1 +
\Theta(1/\kdim))} & = & \Omega( \kdim \log(\pdim - \kdim) ) \; \gg
\Theta(\pdim),
\end{eqnarray*}
so that the scaling $\numobs = \Theta(\pdim)$ is precluded.  In other
concurrent work, Fletcher et al.~\cite{Fletcher08} used direct methods
to show that for the special case of the standard Gaussian ensemble,
the number of observations must satisfy $\numobs >
\Omega\left(\frac{\log (\pdim - \kdim)}{\minval^2} \right)$.  This
bound is a consequence of our lower bound $\myboundtwo$; moreover,
Theorem~\ref{ThmDenseNew} implies the same lower bound for general
(non-Gaussian) ensembles as well.

In the regime of linear sparsity, Wainwright~\cite{Wainwright06_info}
showed, by direct analysis of the optimal decoder, that the scaling
$\minval^2 = \Omega(\log(\kdim)/\kdim)$ is sufficient for exact
support recovery using a linear fraction $\numobs = \Theta(\pdim)$ of
observations.  Combined with the necessary condition in
Theorem~\ref{ThmDenseNew}, we obtain the following corollary that
provides a sharp characterization of the linear-linear regime:
\begin{corollary}%[Sharp conditions for linear-linear recovery]
Consider the regime of linear sparsity, meaning that $\kdim/\pdim =
\alpha \in (0,1)$, and suppose that a linear fraction $\numobs =
\Theta(\pdim)$ of observations are made.  Then the optimal decoder can
recover the support exactly if and only if $\minval^2 = \Omega(\log
\kdim/\kdim)$.
\end{corollary}

\subsection{Effect of measurement sparsity} 

We now turn to the effect of measurement sparsity on recovery,
considering in particular the $\spgam$-sparsified
ensemble~\eqref{EqnGamSparse}.  Even though the average
signal-to-noise ratio of our channel remains the same (since
$\var(X_{ij}) = 1$ for all choices of $\spgam$ by construction), the
Gaussian channel coding bound~\eqref{EqnGaussChannel} is clearly not
tight for sparse $X$, even in the case of $\kdim=1$.  The loss in
statistical efficiency is due to the fact that we are constraining our
codebook to have a sparse structure, which may be far from a
capacity-achieving code. Theorem~\ref{ThmDenseNew} applies to any
ensemble in which the components are zero-mean and unit variance.
However, if we apply it to the $\spgam$-sparsified ensemble, it yields
lower bounds that are independent of $\spgam$.  Intuitively, it is
clear that the procedure of $\spgam$-sparsification should cause
deterioration in support recovery.  Indeed, the following result
provides refined bounds that capture the effects of
$\spgam$-sparsification.  Let $\phi(\mu,\sigma^2)$ denote the Gaussian
density with mean $\mu$ and variance $\sigma^2$, and define the
following two mixture distributions:
\begin{eqnarray}
   \label{EqnDensity1}
   \psibar_1 &\defn& \sum_{l=0}^{\kdim} {\kdim \choose \ell}
   \spgam^\ell (1-\spgam)^{\kdim-\ell} \; \phi\left(0,
   1+\frac{\ell\minval^2}{\spgam}\right)\\
   \label{EqnDensity2}
   \psibar_2 &\defn& \spgam \; \phi\left(0,
   1+\frac{\minval^2}{\spgam}\right) + (1-\spgam) \; \phi(0,1).
\end{eqnarray}
Furthermore, let $H(\cdot)$ denote the entropy functional.  With this
notation, we have the following result.

\begin{theorem}[Sparse ensembles]
\label{ThmSparseNew}
Let the measurement matrix $X \in \real^{\numobs \times \pdim}$ be
drawn with i.i.d. elements from the $\spgam$-sparsified Gaussian
ensemble~\eqref{EqnGamSparse}.  Then a necessary condition for
asymptotically reliable recovery over the signal class
$\SigClass(\betamin)$ is
\begin{eqnarray}
\label{eqn:sparse-converse}
\numobs & > & \max \big \{ \spboundone, \quad \spboundtwo, \quad
\kdim-1 \big \},
\end{eqnarray}
where
\begin{subequations}
\begin{eqnarray}
\spboundone & \defn & \frac{\MYLOGNUM}{H(\psibar_1) -
\frac{1}{2}\log(2\pi\e)} \\
\spboundtwo & \defn & \frac{\log(\pdim - \kdim + 1) - 1}{H(\psibar_2)
- \frac{1}{2}\log(2\pi\e)}.
\end{eqnarray}
\end{subequations}
\end{theorem}

\begin{figure}
\centering \includegraphics[width=4in]{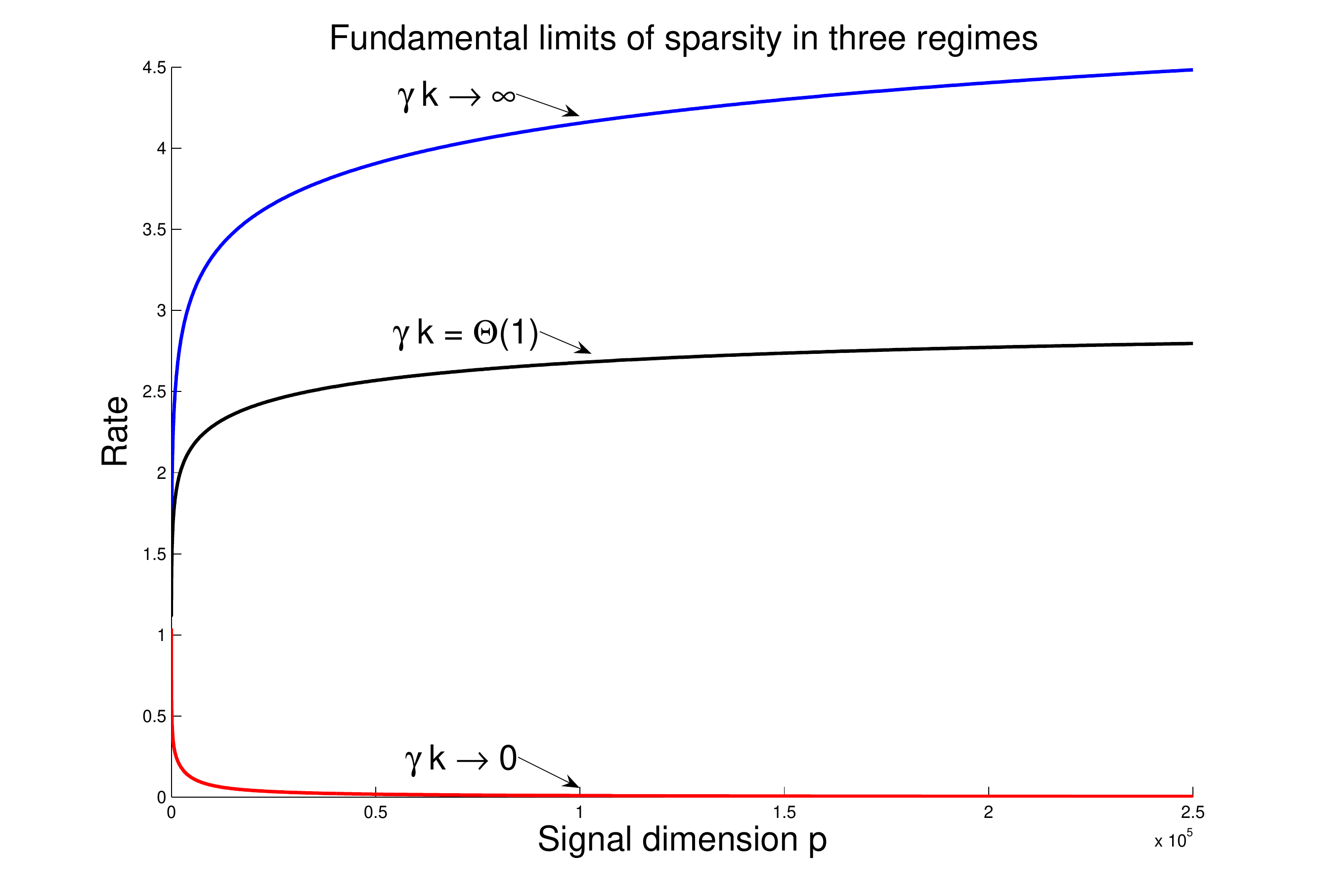}
\caption{The rate $R=\frac{\log {\pdim \choose \kdim}}{\numobs}$ is
plotted using equation~\eqref{eqn:sparse-converse} in three regimes,
depending on how the quantity $\spgam \kdim$ scales, where $\spgam \in
[0,1]$ denotes the measurement sparsification parameter and $\kdim$
denotes the signal sparsity.}  
\label{fig:three-regimes}
\end{figure}

The proof of Theorem~\ref{ThmSparseNew}, given in
Section~\ref{SecProof}, again uses Fano's inequality, but explicitly
analyzes the effect of measurement sparsification on the distribution
of the observations.  The necessary condition in
Theorem~\ref{ThmSparseNew} is plotted in
Figure~\ref{fig:three-regimes}, showing distinct regimes of behavior
depending on how the quantity $\spgam \kdim$ scales, where
$\spgam\in[0,1]$ is the measurement sparsification parameter and
$\kdim$ is the signal sparsity index.  In order to characterize the
thresholds at which measurement sparsity begins to degrade the
performance of any decoder, Corollary~\ref{Cor3Regimes} below further
bounds the necessary conditions in Theorem~\ref{ThmSparseNew} in three
cases.  For any scalar $\spgam$, let $\Hbin(\spgam)$ denote the
entropy of a $\Ber(\spgam)$ variate.

\begin{corollary}[Three regimes]
\label{Cor3Regimes}
The necessary conditions in Theorem~\ref{ThmSparseNew} can be
simplified as follows.
\begin{enumerate}
\item[(a)] In general,
\begin{subequations} \label{EqnRegimeInfty}
\begin{eqnarray} 
   \spboundone &\geq& \frac{\MYLOGNUM}{\frac{1}{2} \log\left(1 + k
   \minval^2\right)}, \\
   \spboundtwo &\geq& \frac{\log(\pdim - \kdim + 1) - 1}{\frac{1}{2}
   \log\left(1 + \minval^2\right)}.
\end{eqnarray}
\end{subequations}
\item[(b)] If $\spgam \kdim = \tau$ for some constant $\tau$, then
\begin{subequations} \label{EqnRegimeConst}
\begin{eqnarray} 
   \spboundone &\geq& \frac{\MYLOGNUM}{\frac{1}{2} \tau \log\left(1 +
   \frac{k\minval^2}{\tau}\right) + C}, \\
   \spboundtwo &\geq& \frac{\log(\pdim - \kdim + 1) - 1}{\frac{1}{2}
   (\frac{\tau}{\kdim}) \log \left( 1 + \frac{\kdim \minval^2}{\tau}
   \right) + \Hbin(\frac{\tau}{\kdim})},
\end{eqnarray}
\end{subequations}
where $C = \frac{1}{2}\log(2\pi\e(\tau + \frac{1}{12}))$ is a
constant.
\item[(c)] If $\spgam \kdim \leq 1$, then
\begin{subequations} \label{EqnRegimeZero}
\begin{eqnarray} 
   \spboundone &\geq& \frac{\MYLOGNUM}{\frac{1}{2} \spgam \kdim
   \log\left(1 + \frac{\minval^2}{\spgam}\right) + \kdim
   \Hbin(\spgam)}, \\
   \spboundtwo &\geq& \frac{\log(\pdim - \kdim + 1) - 1}{\frac{1}{2}
   \spgam \log \left(1 + \frac{\minval^2}{\spgam}\right) +
   \Hbin(\spgam)}.
\end{eqnarray}
\end{subequations}
\end{enumerate}
\end{corollary}

\begin{table} 
\centering
\begin{tabular}{|c|c|c|}
   \hline
   \begin{tabular}{c}
   Necessary conditions \\
   (Theorem~\ref{ThmSparseNew})
   \end{tabular}
   & $\kdim = o(\pdim)$ & $\kdim = \Theta(\pdim)$ \\
   \hline\hline
%%%%%%%%%%%%%%%%%%%%%%%%%%%%%%%%%%%%%%%%%%%%%%%%%%%%%%%%%%%%%%%%%%
   $\begin{array}{c}
   \minval^2 = \Theta(\frac{1}{\kdim}) \\ 
   \spgam = o(\frac{1}{\kdim\log\kdim})
   \end{array}$ 
   & $\Theta\left(\frac{\kdim \log(\pdim-\kdim)}{\spgam\kdim \log
   \frac{1}{\spgam}}\right)$
   & $\Theta\left(\frac{\pdim \log\pdim}{\spgam\pdim \log
   \frac{1}{\spgam}}\right)$ \\
   \hline
   $\begin{array}{c}
   \minval^2 = \Theta(\frac{1}{\kdim}) \\
   \spgam = \Omega(\frac{1}{\kdim\log\kdim}) 
   \end{array}$ 
   & $\Theta(\kdim \log(\pdim-\kdim))$ 
   & $\Theta(\pdim \log\pdim)$ \\
   \hline\hline
%%%%%%%%%%%%%%%%%%%%%%%%%%%%%%%%%%%%%%%%%%%%%%%%%%%%%%%%%%%%%%%%%%
   $\begin{array}{c}
   \minval^2 = \Theta(\frac{\log\kdim}{\kdim}) \\
   \spgam = o(\frac{1}{\kdim\log\kdim}) 
   \end{array}$ 
   & $\Theta\left(\frac{\kdim \log(\pdim-\kdim)}{\spgam\kdim \log
   \frac{1}{\spgam}}\right)$
   & $\Theta\left(\frac{\pdim \log\pdim}{\spgam\pdim \log
   \frac{1}{\spgam}}\right)$ \\
   \hline
   $\begin{array}{c}
   \minval^2 = \Theta(\frac{\log\kdim}{\kdim}) \\
   \spgam = \Theta(\frac{1}{\kdim\log\kdim}) 
   \end{array}$ 
   & $\Theta(\kdim \log(\pdim-\kdim))$ 
   & $\Theta(\pdim \log\pdim)$ \\
   \hline
   $\begin{array}{c}
   \minval^2 = \Theta(\frac{\log\kdim}{\kdim}) \\
   \spgam = \Omega(\frac{1}{\kdim}) 
   \end{array}$
   & $\max \left\{ \Theta\left( \frac{\kdim \log
   \frac{\pdim}{\kdim}}{\log\log\kdim} \right), \Theta\left(
   \frac{\kdim \log(\pdim-\kdim)}{\log\kdim} \right)\right\}$
   & $\Theta(\pdim)$ \\ \hline
\end{tabular}
\caption{Necessary conditions on the number of observations $\numobs$
required for exact support recovery is shown in different regimes of
the parameters $(\pdim, \kdim, \betamin, \spgam)$.}
\label{TabSparse} 
\end{table}

Corollary~\ref{Cor3Regimes} reveals three regimes of behavior, defined
by the scaling of the measurement sparsity $\spgam$ and the signal
sparsity $\kdim$.  If $\spgam \kdim \rightarrow \infty$ as $\pdim
\rightarrow \infty$, then the recovery
threshold~\eqref{EqnRegimeInfty} is of the same order as the threshold
for dense measurement ensembles.  In this regime, sparsifying the
measurement ensemble has no asymptotic effect on performance.  In
sharp contrast, if $\spgam \kdim \rightarrow 0$ sufficiently fast as
$\pdim \rightarrow \infty$, then the recovery
threshold~\eqref{EqnRegimeZero} changes fundamentally compared to the
dense case.  Finally, if $\spgam \kdim = \Theta(1)$, then the recovery
threshold~\eqref{EqnRegimeConst} transitions between the two extremes.
Using the bounds in Corollary~\ref{Cor3Regimes}, the necessary
conditions in Theorem~\ref{ThmSparseNew} are shown in
Table~\ref{TabSparse} under different scalings of the parameters
$(\numobs, \pdim, \kdim, \betamin, \spgam)$.  In particular, if
$\spgam = o(\frac{1}{\kdim\log{\kdim}})$ and the minimum value
$\minval^2$ does not increase with $\kdim$, then the denominator
$\spgam\kdim \log\frac{1}{\spgam}$ goes to zero.  Hence, the number of
measurements that any decoder needs in order to recover reliably
increases dramatically in this regime.

\section{Proofs of our main results}
\label{SecProof}

In this section, we provide the proofs of Theorems~\ref{ThmDenseNew}
and~\ref{ThmSparseNew}.  Establishing necessary conditions for exact
sparsity recovery amounts to finding conditions on $(\numobs, \pdim,
\kdim, \betamin)$ (and possibly $\spgam$) under which the probability
of error of any recovery method stays bounded away from zero as
$\numobs \rightarrow \infty$.  At a high-level, our general approach
is quite simple: we consider restricted problems in which the decoder
has been given some additional side information, and then apply Fano's
inequality~\cite{Cover} to lower bound the probability of error.  In
order to establish the two types of necessary conditions (e.g,
$\myboundone$ versus $\myboundtwo$), we consider two classes of
restricted ensembles: one which captures the bulk effect of having
many competing subsets at large distances, and the other which
captures the effect of a smaller number of subsets at very close
distances.  This is illustrated in Figure~\ref{FigDecodeSubsets}.  We
note that although the first restricted ensemble is a harder problem,
applying Fano to the second restricted ensemble yields a tighter
analysis in some regimes.  In all cases, we assume that the support
$\Sset$ of the unknown vector $\beta \in \R^p$ is chosen randomly and
uniformly over all ${\pdim \choose \kdim}$ possible support sets.
Throughout the remainder of the paper, we use the notation $X_j \in
\R^\numobs$ to denote column $j$ of the matrix $X$, and $X_\Uset \in
\R^{\numobs \times |\Uset|}$ to denote the submatrix containing
columns indexed by set $\Uset$. Similarly, let $\beta_\Uset \in
\R^{|\Uset|}$ denote the subvector of $\beta$ corresponding to the
index set $\Uset$.  \\

\begin{figure}
\centering
\subfloat[]{\label{FigDecodeSubsets}\includegraphics[width=2in]{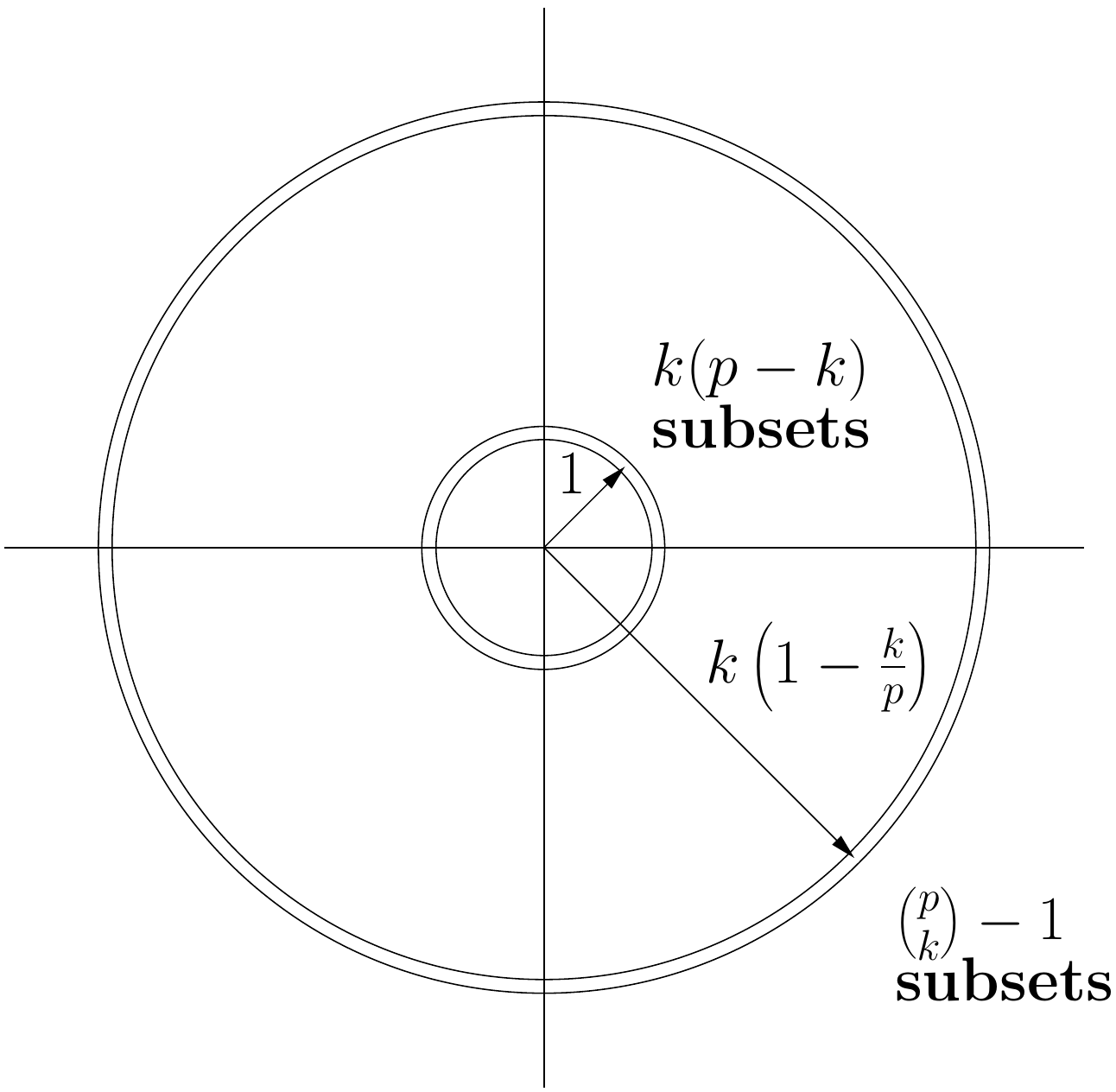}}
\hspace{.75in}
\subfloat[]{\label{FigRestrictedB}\includegraphics[width=2.5in]{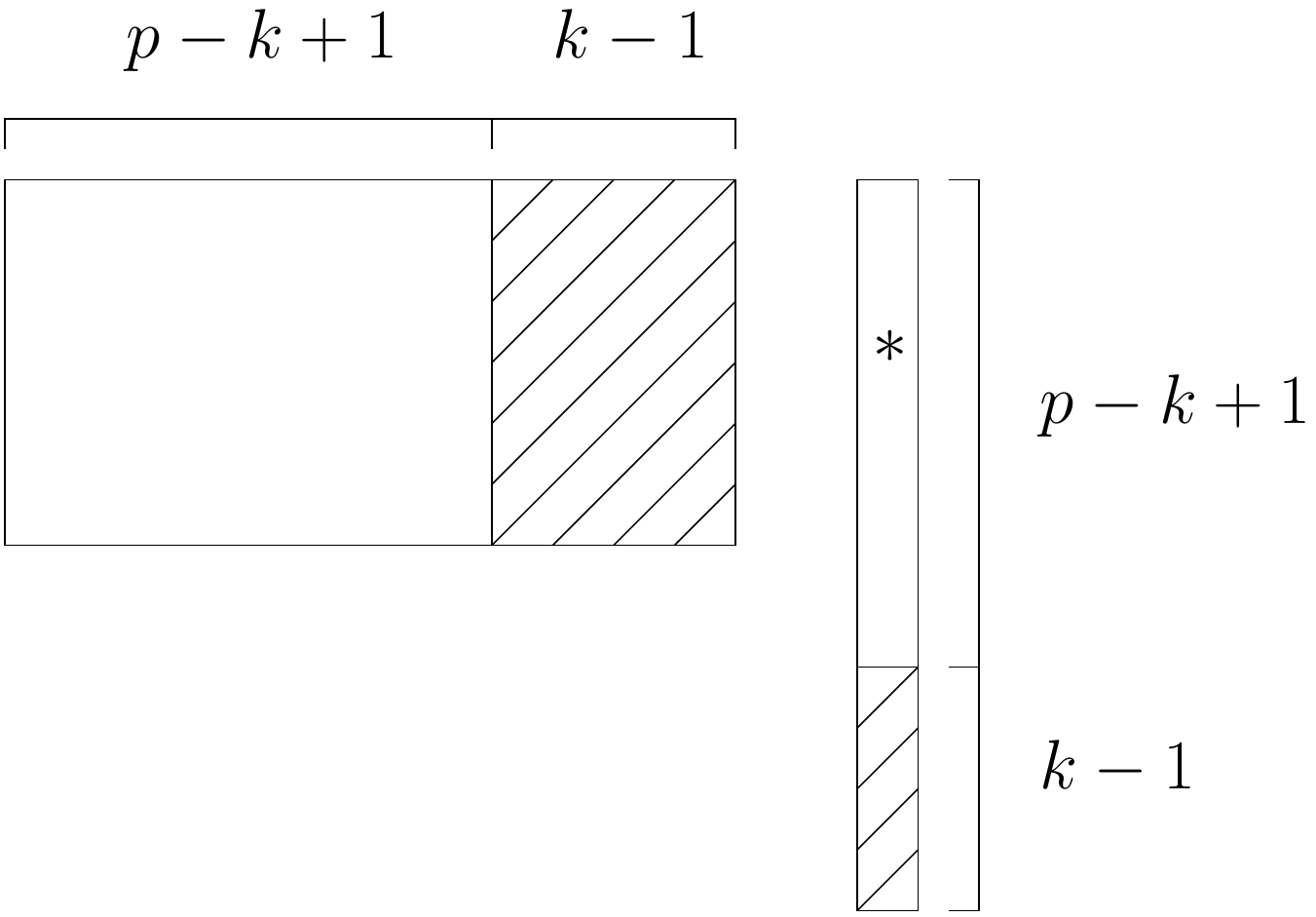}}
\caption{Illustration of restricted ensembles.  (a) In restricted
ensemble A, the decoder must distinguish between ${\pdim \choose
\kdim}$ support sets with an average overlap of size
$\frac{\kdim^2}{\pdim}$, whereas in restricted ensemble B, it must
decode amongst a subset of the $\kdim(\pdim-\kdim)+1$ supports with
overlap $\kdim-1$.  (b) In restricted ensemble B, the decoder is given
the locations of the $\kdim-1$ largest non-zeros, and it must estimate
the location of the smallest non-zero from the $\pdim-\kdim+1$
remaining possible indices.}
\end{figure}

\myparagraph{Restricted ensemble A:} 
In the first restricted problem, also exploited in previous
work~\cite{Wainwright06_info}, we assume that while the support set
$\Sset$ is unknown, the decoder knows \emph{a priori} that $\beta_j =
\minval$ for all $j \in \Sset$.  In other words, the decoder knows the
value of $\beta$ on its support, but it does not know the locations of
the non-zeros.  Conditioned on the event that $\Sset$ is the true
underlying support of $\beta$, the observation vector $Y \in
\R^\numobs$ can then be written as
\begin{eqnarray}
   Y &\defn& \sum_{j \in \Sset} X_j \minval + W.
\end{eqnarray}
If a decoder can recover the support of any $p$-dimensional $k$-sparse
vector $\beta$, then it must be able to recover a $k$-sparse vector
that is constant on its support.  Furthermore, having knowledge of the
value $\minval$ at the decoder cannot increase the probability of
error.  Finally, we assume that $\beta_j = \minval$ for all $j \in
\Sset$ to construct the most difficult possible instance within our
ensemble.  Thus, we can apply Fano's inequality to lower bound the
probability of error in the restricted problem, and so obtain a lower
bound on the probability of error for the general problem.  This
procedure yields the lower bounds $\myboundone$ and $\spboundone$ in
Theorems~\ref{ThmDenseNew} and~\ref{ThmSparseNew} respectively.  \\

\myparagraph{Restricted ensemble B:}
The second restricted ensemble is designed to capture the confusable
effects of the relatively small number $(\pdim - \kdim +1)$ of very
close-by subsets (see Figure~\ref{FigRestrictedB}).  This restricted
ensemble is defined as follows.  Suppose that the decoder is given the
locations of all but the smallest non-zero value of the vector
$\beta$, as well as the values of $\beta$ on its support.  More
precisely, let $\jstar$ denote the unknown location of the smallest
non-zero value of $\beta$, which we assume achieves the minimum (i.e.,
$\beta_{\jstar} = \betamin$), and let $T = S \setminus \{\jstar \}$.
Given knowledge of $(T, \beta_T, \minval)$, the decoder may simply
subtract $X_T \beta_T = \sum_{j \in T} X_j \beta_j$ from $Y$, so that
it is left with the modified $\numobs$-vector of observations
\begin{eqnarray}
\label{EqnModObs}
   \Ytil & \defn & X_{\jstar} \minval + W.
\end{eqnarray}
By re-ordering indices as need be, we may assume without loss of
generality that $T = \{\pdim-\kdim+2, \ldots, \pdim\}$, so that
\mbox{$\jstar \in \{1, \ldots, p-k+1\}$.}  The remaining sub-problem
is to determine, given the observations $\Ytil$, the location of the
single non-zero.  Note that when we assume that the support of $\beta$
is uniformly chosen over all ${\pdim \choose \kdim}$ possible subsets
of size $\kdim$, then given $T$, the location of the remaining
non-zero is uniformly distributed over $\{1, \ldots, p-k+1\}$.

We will now argue that analyzing the probability of error of this
restricted problem gives us a lower bound on the probability of error
in the original problem.  Let $\betatil \in \R^{p-k+1}$ be a vector
with exactly one non-zero.  We can augment $\betatil$ with $k-1$
non-zeros at the end to obtain a $p$-dimensional vector.  If a decoder
can recover the support of any $p$-dimensional $k$-sparse vector
$\beta$, then it can recover the support of the augmented $\betatil$,
and hence the support of $\betatil$.  Similarly, providing the decoder
with side information about the non-zero values of $\beta$ cannot
increase the probability of error.  As before, we can apply Fano's
inequality to lower bound the probability of error in this restricted
problem, thereby obtaining the lower bounds $\myboundtwo$ and
$\spboundtwo$ in Theorems~\ref{ThmDenseNew} and~\ref{ThmSparseNew}
respectively.

\subsection{Proof of Theorem~\ref{ThmDenseNew}}

In this section, we derive the necessary conditions $\myboundone$ and
$\myboundtwo$ in Theorem~\ref{ThmDenseNew} for the general class of
measurement matrices, by applying Fano's inequality to bound the
probability of decoding error in restricted problems A and B,
respectively.  \\

\subsubsection{Applying Fano to restricted ensemble A}

We first perform our analysis of the error probability for a
particular instance of the random measurement matrix $X$, and
subsequently average over the ensemble of matrices.  Let $\Omega$
denote a random subset chosen uniformly at random over all ${\pdim
\choose \kdim}$ subsets $\Sset \subset \{1, \ldots, p\}$ of size
$\kdim$.  The probability of decoding error, for a given $X$, can be
lower bounded by Fano's inequality as
\begin{eqnarray*}   
   p_{err}(X) &\geq& \frac{H(\Omega | Y) - 1}{\log{\pdim \choose \kdim}}
   \;\; = \;\; 1 - \frac{I(\Omega;Y) + 1}{\log{\pdim \choose \kdim}}
\end{eqnarray*}
where we have used the fact that $H(\Omega|Y) = H(\Omega)-I(\Omega;Y)
= \log{\pdim \choose \kdim}-I(\Omega;Y)$.  Thus the problem is reduced
to upper bounding the mutual information $I(\Omega;Y)$ between the
random subset $\Omega$ and the noisy observations $Y$.  Since both $X$
and $\minval$ are known and fixed, the mutual information can be
expanded as
\begin{eqnarray*}
   I(\Omega;Y) &=& H(Y) - H(Y | \Omega)
   \;\; = \;\; H(Y) - H(W).
\end{eqnarray*}
We first bound the entropy of the observation vector $H(Y)$, using the
fact that differential entropy is maximized by the Gaussian
distribution with a matched variance.  More specifically, for a given
$X$, let $\Lambda(X)$ denote the covariance matrix of $Y$ conditioned
on $X$.  (Hence entry $\Lambda_{ii}(X)$ on the diagonal represents the
variance of $Y_i$.)  With this notation, the entropy of $Y$ can be
bounded as
\begin{eqnarray*}
   H(Y) &\leq& \sum_{i=1}^n H(Y_i) \\
   &\leq& \sum_{i=1}^n \frac{1}{2} \log(2\pi\e \; \Lambda_{ii}(X)).
\end{eqnarray*}
Next, the entropy of the Gaussian noise vector $W \sim N(0,I_{n \times
n})$ can be computed as $H(W) = \frac{n}{2} \log(2\pi\e)$.  Combining
these two terms, we then obtain the following bound on the mutual
information,
\begin{eqnarray*}
   I(\Omega;Y) &\leq& \sum_{i=1}^n \frac{1}{2} \log(\Lambda_{ii}(X)).
\end{eqnarray*}

With this bound on the mutual information, we now average the
probability of error over the ensemble of measurement matrices $X$.
Exploiting the concavity of the logarithm and applying Jensen's
inequality, the average probability of error can be bounded as
\begin{eqnarray} \label{EqnPerr}
   \Exs_X[p_{err}(X)] &\geq& 1 - \frac{\sum_{i=1}^n \frac{1}{2}
   \log(\Exs_X[\Lambda_{ii}(X)]) + 1}{\log{\pdim \choose \kdim}}.
\end{eqnarray}

It remains to compute the expectation $\Exs_X[\Lambda_{ii}(X)]$,
over the ensemble of matrices $X$ drawn with i.i.d. entries from any
distribution with zero-mean and unit variance.  The proof of the
following lemma involves some relatively straightforward but lengthy
calculation, and is given in Appendix~\ref{PfLemMoments}.
\begin{lemma}\label{LemMoments}
Given i.i.d. $X_{ij}$ with zero-mean and unit variance, the average
covariance is given by
\begin{eqnarray}
   \Exs_X[\Lambda(X)] &=& \left(1 + \kdim \minval^2 \left(1 -
   \frac{\kdim}{\pdim}\right)\right) I_{\numobs\times\numobs}.
\end{eqnarray}
\end{lemma}

Finally, combining Lemma~\ref{LemMoments} with
equation~\eqref{EqnPerr}, we obtain that the average probability
of error is bounded away from zero if
\begin{eqnarray*}
   n & < &\frac{\log{p \choose k}-1}{\frac{1}{2} \log\left(k\minval^2
   \left(1 - \frac{k}{p}\right) + 1\right)},
\end{eqnarray*}
as claimed.
\\

\subsubsection{Applying Fano to restricted ensemble B}

The analysis of restricted ensemble B is completely analogous to the
proof for restricted ensemble A, so we will only outline the key steps
below.  Let $\Omega$ denote a random variable with uniform
distribution over the indices $\{1, \ldots, \pdim-\kdim+1\}$.  The
probability of decoding error, for a given measurement matrix $X$, can
be lower bounded by Fano's inequality as
\begin{eqnarray*}
   p_{err}(X) &\geq& 1 - \frac{I(\Omega;\Ytil) + 1}{\log(\pdim-\kdim+1)}
\end{eqnarray*}
As before, the key problem of bounding the mutual information
$I(\Omega;\Ytil)$ between the random index $\Omega$ and the modified
observation vector $\Ytil$, can be reduced to bounding the entropy
$H(\Ytil)$.  For each fixed $X$, let $\Lambda(X)$ denote the
covariance matrix of $\Ytil$.  Since the differential entropy of
$\Ytil_i$ is upper bounded by the entropy of a Gaussian distribution
with variance $\Lambda_{ii}(X)$, we obtain the following bound on the
mutual information
\begin{eqnarray*}
   I(\Omega;\Ytil) &=& H(\Ytil) - \frac{n}{2} \log(2\pi\e) \\
   &\leq& \sum_{i=1}^n \frac{1}{2} \log(\Lambda_{ii}(X)).
\end{eqnarray*}
Applying Jensen's inequality, we can then bound the average
probability of error, averaged over the ensemble of measurement
matrices $X$, as
\begin{eqnarray} \label{EqnPerrRP2}
   \Exs_X[p_{err}(X)] &\geq& 1 - \frac{\sum_{i=1}^n \frac{1}{2}
   \log(\Exs_X[\Lambda_{ii}(X)]) + 1}{\log(\pdim-\kdim+1)}.
\end{eqnarray}
The proof of Lemma~\ref{LemMomentsRP2} below follows the same steps as
the derivation of Lemma~\ref{LemMoments}, and is omitted.
\begin{lemma}\label{LemMomentsRP2}
Given i.i.d. $X_{ij}$ with zero-mean and unit variance, the average
covariance is given by
\begin{eqnarray}
   \Exs_X[\Lambda(X)] &=& \left(1 + \minval^2 \left(1 -
   \frac{1}{\pdim-\kdim+1}\right)\right) I_{\numobs\times\numobs}.
\end{eqnarray}
\end{lemma}

Finally, combining Lemma~\ref{LemMomentsRP2} with the Fano
bound~\eqref{EqnPerrRP2}, we obtain that the average probability of
error is bounded away from zero if
\begin{eqnarray*}
   n &<& \frac{\log (\pdim - \kdim +1) - 1}{\frac{1}{2} \log(1+
\minval^2 (1-\frac{1}{\pdim-\kdim+1}))}
\end{eqnarray*}
as claimed.

\subsection{Proof of Theorem~\ref{ThmSparseNew}}

This section contains proofs of the necessary conditions in
Theorem~\ref{ThmSparseNew} for the $\spgam$-sparsified Gaussian
measurement ensemble~\eqref{EqnGamSparse}.  We proceed as before,
applying Fano's inequality to restricted problems A and B, in order to
derive the conditions $\spboundone$ and $\spboundtwo$, respectively.
\\

\subsubsection{Analyzing restricted ensemble A}

In analyzing the probability of error in restricted ensemble A, the
initial steps proceed as in the proof of Theorem~\ref{ThmDenseNew},
first bounding the probability of error for a fixed instance of the
measurement matrix $X$, and later averaging over the
$\spgam$-sparsified Gaussian ensemble~\eqref{EqnGamSparse}.  Let
$\Omega$ denote a random subset uniformly distributed over the ${\pdim
\choose \kdim}$ possible subsets $\Sset \subset \{1, \ldots, p\}$ of
size $\kdim$.  As before, the probability of decoding error, for each
fixed $X$, can be lower bounded by Fano's inequality as
\begin{eqnarray*}
   p_{err}(X) &\geq& 1 - \frac{I(\Omega;Y)+1}{\log{p \choose k}}.
\end{eqnarray*}
We can similarly bound the mutual information
\begin{eqnarray*}
   I(\Omega;Y) &=& H(Y) - H(W)\\ 
   &\leq& \sum_{i=1}^n H(Y_i) - \frac{n}{2}\log(2\pi\e),
\end{eqnarray*}
using the Gaussian entropy for $W \sim N(0, I_{\numobs\times\numobs})$.

From this point, the key subproblem is to compute the entropy of $Y_i
=\sum_{j\in\Sset} X_{ij} \minval + W_i$.  To characterize the limiting
behavior of the random variable $Y_i$, note that $Y_i$ is distributed
according to the density defined as
\begin{eqnarray*}
   \psi_1(y,i; X) &=& \frac{1}{{p \choose k}} \sum_{S}
   \frac{1}{\sqrt{2\pi}} \exp{\left(-\frac{1}{2} (y-\minval \sum_{j
   \in S} X_{ij})^2\right)}.
\end{eqnarray*}
For each fixed matrix $X$, this density is a mixture of Gaussians with
unit variances and means that depend on the values of
$\{X_{i1},\ldots,X_{ip}\}$, summed over subsets $S \subset
\{1,\ldots,p\}$ with $|S|=k$.  At a high-level, our immediate goal is
to characterize the entropy $H(\psi_1)$.

Note that as $X$ varies over the ensemble~\eqref{EqnGamSparse}, the
sequence $\{\psi_1(\cdot \,; X) \}_p$, indexed by the signal dimension
$\pdim$, is actually a sequence of random densities.  As an
intermediate step, the following lemma characterizes the average
pointwise behavior of this random sequence of densities, and is proven
in Appendix~\ref{PfLemUstat}.
\begin{lemma}\label{LemUstat}
Let $X$ be drawn with i.i.d. entries from the $\spgam$-sparsified
Gaussian ensemble~\eqref{EqnGamSparse}.  For each fixed $y$ and for
all $i = 1,\ldots,\numobs$, $\; \Exs_X[\psi_1(y,i;X)] = \psibar_1(y)$, where
\begin{eqnarray} \label{EqnUstat}
   \psibar_1(y) &=& \Exs_L \left[ \frac{1}{\sqrt{2\pi(1+\frac{L
   \minval^2}{\spgam})}} \exp{\left(-\frac{y^2}{2 (1+\frac{L
   \minval^2}{\spgam}) }\right)} \right]
\end{eqnarray}
is a mixture of Gaussians with binomial weights \mbox{$L \sim
\Bin(\kdim,\spgam)$.}
\end{lemma}
For certain scalings, we can use concentration results for
$U$-statistics~\cite{Serfling80} to prove that $\psi_1$ converges
uniformly to $\psibar_1$, and from there that $H(\psi_1)
\stackrel{p}{\rightarrow} H(\psibar_1)$.  In general, however, we
always have an upper bound, which is sufficient for our purposes.
Indeed, since differential entropy $H(\psi_1)$ is a concave function
of $\psi_1$, by Jensen's inequality and Lemma~\ref{LemUstat}, we have
\begin{eqnarray*}
   \Exs_X [H(\psi_1)] & \leq & H \left(\Exs_X[\psi_1] \right) \; = \;
   H(\psibar_1).
\end{eqnarray*}

With these ingredients, we conclude that the average error probability
of any decoder, averaged over the sparsified Gaussian measurement
ensemble, is lower bounded by
\begin{eqnarray*}
   \Exs_X[p_{err}(X)] & \geq & 1 - \frac{\sum_{i=1}^n \Exs_X[H(Y_i)] -
   \frac{n}{2} \log(2\pi\e) + 1}{\log{p \choose k}}\\
   &=& 1 - \frac{\sum_{i=1}^n \Exs_X[H(\psi_1)] - \frac{n}{2} \log(2\pi\e) +
   1}{\log{p \choose k}}\\
   &\geq& 1 - \frac{n H(\psibar_1) - \frac{n}{2} \log(2\pi\e) +
   1}{\log{p \choose k}}.
\end{eqnarray*}
Therefore, the probability of decoding error is bounded away from zero
if
\begin{eqnarray*}
   \numobs &<& \frac{\log{\pdim \choose \kdim} - 1}{H(\psibar_1) -
   \frac{1}{2} \log(2\pi\e)},
\end{eqnarray*}
as claimed.
\\

\subsubsection{Analyzing restricted ensemble B}

The analysis of restricted ensemble B mirrors exactly the derivation
of restricted ensemble A.  Hence we only outline the key steps in this
section.  Letting $\Omega \sim \operatorname{Uni}\{1, \ldots,
\pdim-\kdim+1\}$, we again apply Fano's inequality to restricted
problem B, using the sparse measurement ensemble~\eqref{EqnGamSparse}:
\begin{eqnarray*}
   p_{err}(X) &\geq& 1 - \frac{I(\Omega;\Ytil) + 1}{\log(\pdim-\kdim+1)}.
\end{eqnarray*}
In order to upper bound $I(\Omega;\Ytil)$, we need to upper bound the
entropy $H(\Ytil)$.  The sequence of densities associated with
$\Ytil_i$ becomes
\begin{eqnarray*}
   \psi_2(y,i;X) &=& \frac{1}{\pdim-\kdim+1}
   \sum_{j=1}^{\pdim-\kdim+1} \frac{1}{\sqrt{2\pi}}
   \exp\left(-\frac{1}{2} (y - \minval X_{ij})^2\right).
\end{eqnarray*}
Lemma~\ref{LemUstatRP2} below characterizes the average pointwise
behavior of these densities, and follows from the proof of
Lemma~\ref{LemUstat}, with $\Sset$ taken to be subsets of the indices
$\{1,\ldots,\pdim-\kdim+1\}$ of size $|S|=1$.
\begin{lemma}\label{LemUstatRP2}
Let $X$ be drawn with i.i.d. entries according
to~\eqref{EqnGamSparse}.  For each fixed $y$ and for all
$i=1,\ldots,\numobs$, $\; \Exs_X[\psi_2(y,i;X)] = \psibar_2(y)$, where
\begin{eqnarray} \label{EqnUstatRP2}
   \psibar_2(y) &=& \Exs_B\left[\frac{1}{\sqrt{2\pi
   (1+\frac{B\minval^2}{\spgam})}} \exp \left(-\frac{y^2}{2
   (1+\frac{B\minval^2}{\spgam})}\right)\right]
\end{eqnarray}
is a mixture of Gaussians with Bernoulli weights $B \sim
\Ber(\spgam)$.
\end{lemma}
As before, we can apply Jensen's inequality to obtain the bound
\begin{equation*}
\Exs_X[H(\psi_2)] \; \leq \; H(\Exs_X[\psi_2]) \; = \; H(\psibar_2).
\end{equation*}
The necessary condition then follows by the Fano bound on the
probability of error.

\subsection{Proof of Corollary~\ref{Cor3Regimes}}

In this section, we derive bounds on the expressions $\spboundone$ and
$\spboundtwo$ in Theorem~\ref{ThmSparseNew}.  We begin by noting that
the Gaussian mixture distribution $\psibar_1$ defined
in~\eqref{EqnDensity1} is a strict generalization of the distribution
$\psibar_2$ defined in~\eqref{EqnDensity2}; moreover, setting the
parameter $\kdim=1$ in $\psibar_1$ recovers $\psibar_2$.  The variance
associated with the mixture distribution $\psibar_1$ is equal to
$\sigma_1^2 = 1 + \kdim\minval^2$, and so the entropy of $\psibar_1$
is always bounded by the entropy of a Gaussian distribution with
variance $\sigma_1^2$, as
\begin{eqnarray*}
   H(\psibar_1) &\leq& \frac{1}{2} \log(2\pi\e (1 + \kdim\minval^2)).
\end{eqnarray*}
Similarly, the mixture distribution $\psibar_2$ has variance equal to
$1 + \minval^2$, so that the entropy associated with $\psibar_2$ can
in general be bounded as
\begin{eqnarray*}
   H(\psibar_2) &\leq& \frac{1}{2} \log(2\pi\e (1 + \minval^2)).
\end{eqnarray*}
This yields the first set of bounds in~\eqref{EqnRegimeInfty}.

Next, to derive more refined bounds which capture the effects of
measurement sparsity, we will make use of the following lemma (which
is proven in Appendix~\ref{PfLemEntropy}) to bound the entropy
associated with the mixture distribution $\psibar_1$:
\begin{lemma}\label{LemEntropy}
For the Gaussian mixture distribution $\psibar_1$ defined
in~\eqref{EqnDensity1},
\begin{eqnarray*}
   H(\psibar_1) &\leq& \Exs_L\left[\frac{1}{2} \log\left( 2\pi\e
   \left(1+\frac{L\minval^2}{\spgam}\right) \right)\right] + H(L),
\end{eqnarray*}
where $L \sim \Bin(\kdim,\spgam)$.
\end{lemma}
We can further bound the expression in Lemma~\ref{LemEntropy} in three
cases, delineated by the quantity $\spgam\kdim$.  The proof of the
following claim in given in Appendix~\ref{PfLemExsBin}.
\begin{lemma} \label{LemExsBin}
Let $E = \Exs_L \left[ \frac{1}{2} \log \left( 1 +
\frac{L\minval^2}{\spgam} \right) \right]$, where $L \sim
\Bin(\kdim,\spgam)$.
\begin{enumerate}
  \item[(a)] If $\spgam\kdim > 3$, then
  \begin{eqnarray*}
    \frac{1}{4} \log\left(1 + \frac{\kdim\minval^2}{3}\right) &\leq& E
    \;\;\leq\;\; \frac{1}{2} \log\left(1+\kdim\minval^2\right).
  \end{eqnarray*}
  \item[(b)] If $\spgam\kdim = \tau$ for some constant $\tau$, then
  \begin{eqnarray*}
    \frac{1}{2} (1 - \e^{-\tau}) \log\left(1 +
    \frac{\kdim\minval^2}{\tau}\right) &\leq& E \;\;\leq\;\;
    \frac{1}{2} \tau \log\left(1+\frac{\kdim\minval^2}{\tau}\right).
  \end{eqnarray*}
  \item[(c)] If $\spgam\kdim \leq 1$, then
  \begin{eqnarray*}
    \frac{1}{4} \spgam\kdim \log\left(1 + \frac{\minval^2}{\spgam}
    \right) &\leq& E \;\;\leq\;\; \frac{1}{2} \spgam\kdim \log\left(1
    + \frac{\minval^2}{\spgam}\right).
  \end{eqnarray*}
\end{enumerate}
\end{lemma}
Finally, combining Lemmas~\ref{LemEntropy} and~\ref{LemExsBin} with
some simple bounds on the entropy of the binomial variate $L$ (given
in Appendix~\ref{AppBinEntropy}), we obtain the bounds on
$\spboundone$ in~\eqref{EqnRegimeConst} and~\eqref{EqnRegimeZero}.

We can similarly bound the entropy associated with the Gaussian
mixture distribution $\psibar_2$.  Since the density $\psibar_2$ is a
special case of the density $\psibar_1$ with $\kdim$ set to $1$, we
can again apply Lemma~\ref{LemEntropy} to obtain
\begin{eqnarray*}
   H(\psibar_2) &\leq& \Exs_B\left[\frac{1}{2} \log\left(2\pi\e
   \left(1+\frac{B\minval^2}{\spgam}\right)\right)\right] + H(B)\\
   &=& \frac{\spgam}{2} \log\left(1 + \frac{\minval^2}{\spgam}\right)
   + \Hbin(\spgam) + \frac{1}{2}\log(2\pi\e).
\end{eqnarray*}
We have thus obtained the bounds on $\spboundtwo$ in
equations~\eqref{EqnRegimeConst} and~\eqref{EqnRegimeZero}.

\section{Discussion}
\label{SecDiscuss}

In this paper, we have studied the information-theoretic limits of
exact support recovery for general scalings of the parameters
$(\numobs, \pdim, \kdim, \betamin, \spgam)$.  Our first result
(Theorem~\ref{ThmDenseNew}) applies generally to measurement matrices
with zero-mean and unit variance entries.  It strengthens previously
known bounds, and combined with known sufficient
conditions~\cite{Wainwright06_info}, yields a sharp characterization
of recovering signals with linear sparsity with a linear fraction of
observations (Corollary~\ref{Cor3Regimes}).  Our second result
(Theorem~\ref{ThmSparseNew}) applies to $\spgam$-sparsified Gaussian
measurement ensembles, and reveals three different regimes of
measurement sparsity, depending on how significantly they impair
statistical efficiency.  For linear signal sparsity,
Theorem~\ref{ThmSparseNew} is not a sharp result (by comparison to
Theorem~\ref{ThmDenseNew} in the dense case); however, its tightness
for sublinear signal sparsity is an interesting open problem.
Finally, Theorem~\ref{ThmDenseNew} implies that the standard Gaussian
ensemble is an information-theoretically optimal choice for the
measurement matrix: no other zero-mean unit variance distribution can
reduce the number of observations necessary for recovery, and in fact
the standard Gaussian distribution achieves matching sufficient
bounds~\cite{Wainwright06_info}.  This fact raises an interesting open
question on the design of other, more computationally friendly,
measurement matrices which are optimal in the information-theoretic
sense.

\subsection*{Acknowledgment}  The work of WW and KR was
supported by NSF grant CCF-0635114. The work of MJW was supported by
NSF grants CAREER-CCF-0545862 and DMS-0605165.

\appendix

\section{Proof of Lemma \ref{LemMoments}}
\label{PfLemMoments}

We begin by defining some additional notation.  Let $\betabar \in
\R^{\pdim}$ be a $\kdim$-sparse vector with $\betabar_j = \minval$ for
all indices $j$ in the support set $\Sset$.  Recall that $\Omega$
denotes a random subset uniformly distributed over all ${\pdim \choose
\kdim}$ possible subsets $\Sset \subset \{1, \ldots, \pdim\}$ with
$|\Sset|=\kdim$.  Conditioned on the event that $\Omega = \Sset$, the
vector of $\numobs$ observations can then be written as
\begin{eqnarray*}
   \Ybar &\defn& X_\Sset \betabar_\Sset + W 
   \;\; = \;\; \minval \sum_{j \in \Sset} X_j + W.
\end{eqnarray*}
Note that for a given instance of the matrix $X$, the distribution of
$\Ybar$ is a Gaussian mixture with density $f(y) = \myprob{1}
\sum_\Sset \phi(X_\Sset \betabar_\Sset, I)$, where we are using $\phi$
to denote the density of a Gaussian random vector with mean $X_\Sset
\betabar_\Sset$ and covariance $I$.  Let $\mu(X) = \mu \in
\R^{\numobs}$ and $\Lambda(X) = \Lambda \in \R^{\numobs \times
\numobs}$ be the mean vector and covariance matrix of $\Ybar$,
respectively.  The covariance matrix of $\Ybar$ can be computed as
$\Lambda = \Exs\left[\Ybar \Ybar^T\right] - \mu \mu^T$, where
\begin{eqnarray*}
   \mu &=& \myprob{1} \sum_\Sset X_\Sset \betabar_\Sset 
   \;\; = \;\; \myprob{\minval} \sum_\Sset \sum_{j\in\Sset} X_j
\end{eqnarray*}
and
\begin{eqnarray*}
   \Exs\left[\Ybar \Ybar^T\right] &=& \Exs\left[(X \betabar)
   (X \betabar)^T \right] + \Exs\left[W W^T\right]\\
   &=& \myprob{1} \sum_\Sset (X_\Sset \betabar_\Sset) (X_\Sset
   \betabar_\Sset)^T \; + \; I.
\end{eqnarray*}

With this notation, we can now compute the expectation of the
covariance matrix $\Exs_X\big[\Lambda\big]$, averaged over any
distribution on $X$ with independent, zero-mean and unit variance
entries.  To compute the first term, we have
\begin{eqnarray*}
   \Exs_X\left[ \Exs\left[\Ybar \Ybar^T\right] \right] &=&
   \myprob{\minval^2} \sum_\Sset \Exs_X\left[ \sum_{j \in \Sset} X_j
   X_j^T + \sum_{i \neq j \in \Sset} X_i X_j^T \right] \;\; + \;\; I
   \\
   &=& \myprob{\minval^2} \sum_\Sset \sum_{j \in \Sset} I \;\; + \;\; I \\
   &=& \left(1 + \kdim \minval^2\right) \; I
\end{eqnarray*}
where the second equality uses the fact that $\Exs_X\left[ X_j X_j^T
\right] = I$, and $\Exs_X\left[ X_i X_j^T \right] = 0$ for $i \neq j$.
Next, we compute the second term as,
\begin{eqnarray*}
   \Exs_X \left[ \mu \mu^T \right] &=& \left(\myprob{\minval}\right)^2
   \Exs_X \left[ \sum_{\Sset,\Uset} \sum_{j\in\Sset\cap\Uset} X_j
   X_j^T + \sum_{\Sset,\Uset} \sum_{i\in\Sset,j\in\Uset \atop i \neq
   j} X_i X_j^T \right] \\
   &=& \left(\myprob{\minval}\right)^2 \sum_{\Sset,\Uset}
   \sum_{j\in\Sset\cap\Uset} I \\
   &=& \left(\left(\myprob{\minval}\right)^2 \sum_{\Sset,\Uset} |\Sset
   \cap \Uset|\right) \; I.
\end{eqnarray*}
From here, note that there are ${\pdim \choose \kdim}$ possible
subsets $\Sset$.  For each $\Sset$, a counting argument reveals that
there are ${\kdim \choose \lambda} {\pdim-\kdim \choose
\kdim-\lambda}$ subsets $\Uset$ of size $\kdim$ which have $\lambda =
|\Sset \cap \Uset|$ overlaps with $\Sset$.  Thus the scalar
multiplicative factor above can be written as
\begin{eqnarray*}
   \left(\myprob{\minval}\right)^2 \sum_{\Sset,\Uset} |\Sset \cap
   \Uset| &=& \myprob{\minval^2} \sum_{\lambda=1}^\kdim {\kdim \choose
   \lambda} {\pdim-\kdim \choose \kdim-\lambda} \lambda.
\end{eqnarray*}
Finally, using a substitution of variables (by setting $\lambda' =
\lambda-1$) and applying Vandermonde's identity~\cite{Riordan}, we
have
\begin{eqnarray*}
   \left(\myprob{\minval}\right)^2 \sum_{\Sset,\Uset} |\Sset \cap
   \Uset| &=& \myprob{\minval^2} \; \kdim \sum_{\lambda'=0}^{\kdim-1}
   {\kdim-1 \choose \lambda'} {\pdim-\kdim \choose \kdim-\lambda'-1}
   \\
   &=& \myprob{\minval^2} \; \kdim {\pdim-1 \choose \kdim-1} \\
   &=& \frac{\kdim^2 \minval^2}{\pdim}.
\end{eqnarray*}

Combining these terms, we conclude that
\begin{eqnarray*}
   \Exs_X\big[\Lambda(X) \big] &=& \left(1 + \kdim \minval^2
   \left(1-\frac{\kdim}{\pdim}\right) \right) \; I.
\end{eqnarray*}

\section{Proof of Lemma \ref{LemUstat}}
\label{PfLemUstat}

Consider the following sequences of densities,
\begin{eqnarray*}
   \psi_1(y,i; X) &=& \frac{1}{{p \choose k}} \sum_{S}
   \frac{1}{\sqrt{2\pi}} \exp{\left(-\frac{1}{2} (y-\minval \sum_{j
   \in S} X_{ij})^2\right)}\\
\end{eqnarray*}
and
\begin{eqnarray*}
   \psibar_1(y) &=& \Exs_L \left[ \frac{1}{\sqrt{2\pi(1+\frac{L
   \minval^2}{\gamma})}} \exp{\left(-\frac{y^2}{2 (1+\frac{L
   \minval^2}{\gamma}) }\right)} \right]
\end{eqnarray*}
where $L \sim \Bin(\kdim, \gamma)$.  Our goal is to show that for each
fixed $y$ and row index $i$, the pointwise average of the stochastic
sequence of densities $\psi_1$ over the ensemble of matrices $X$
satisfies $\Exs_X[\psi_1(y,i;X)] = \psibar_1(y)$.  By symmetry, it is
sufficient to compute this expectation for the subset $\Sset = \{1,
\ldots, \kdim\}$.  When each $X_{ij}$ is i.i.d. drawn according to the
$\spgam$-sparsified ensemble~\eqref{EqnGamSparse}, the random variable
$Z = (y - \minval \sum_{j=1}^k X_{ij})$ has a Gaussian mixture
distribution which can be described as follows.  Denoting the mixture
label by $L \sim \Bin(\kdim, \gamma)$, then $Z \sim
N\left(y,\frac{\ell \minval^2}{\gamma}\right)$ if $L=\ell$, for
$\ell=0,\ldots,k$.  Thus, conditioned on the mixture label $L=\ell$,
the random variable $\tilde{Z} = \frac{\gamma}{\ell \minval^2} (y -
\minval \sum_{j=1}^k X_{ij})^2$ has a noncentral chi-square
distribution with $1$ degree of freedom and parameter $\lambda =
\frac{\gamma y^2}{\ell \minval^2}$.  Evaluating $M_{\tilde{Z}}(t) =
\Exs[\e^{t \tilde{Z}}]$, the moment-generating function~\cite{Birge01}
of $\tilde{Z}$, then gives us the desired quantity,
\begin{eqnarray*}
   && \Exs_X\left[\frac{1}{\sqrt{2\pi}} \exp\left(-\frac{1}{2} (y -
   \minval \sum_{j=1}^k X_{ij})^2\right)\right] \\
   && = \sum_{\ell=0}^k \frac{1}{\sqrt{2\pi}} \Exs_X\left[\left. \exp
   \left(-\frac{1}{2} (y - \minval \sum_{j=1}^k X_{ij})^2\right) \;
   \right| \; L = \ell \right] \mprob(L=\ell)\\
   && = \sum_{\ell=0}^k \frac{1}{\sqrt{2\pi}} M_{\tilde{Z}}
   \left(-\frac{\ell \minval^2}{2 \gamma}\right) \mprob(L=\ell) \\
   && = \Exs_L \left[ \frac{1}{\sqrt{2\pi (1 +
   \frac{L\minval^2}{\gamma})}} \exp{\left(-\frac{y^2}{2 (1 +
   \frac{L\minval^2}{\gamma}) }\right)} \right]
\end{eqnarray*}
as claimed.

\section{Proof of Lemma \ref{LemEntropy}}
\label{PfLemEntropy}

Let $Z$ be a random variable distributed according to the density
\begin{eqnarray*}
   \psibar_1(y) &=& \Exs_L \left[ \frac{1}{\sqrt{2\pi (1 +
   \frac{L\minval^2}{\gamma})}} \exp \left(-\frac{y^2}{2(1 +
   \frac{L\minval^2}{\gamma})} \right) \right],
\end{eqnarray*}
where $L \sim Bin(k,\gamma)$.  To compute the entropy of $Z$, we can
expand the following mutual information in two ways, $I(Z;L) = H(Z) -
H(Z|L) = H(L) - H(L|Z)$, and obtain
\begin{eqnarray*}
   H(Z) &=& H(Z|L) + H(L) - H(L|Z).
\end{eqnarray*}
The conditional distribution of $Z$ given that $L=\ell$ is Gaussian,
and so the conditional entropy of $Z$ given $L$ can be written as
\begin{eqnarray*}
   H(Z|L) &=& \Exs_L\left[ \frac{1}{2} \log\left( 2\pi\e \left(1 +
   \frac{L \minval^2}{\gamma}\right) \right) \right].
\end{eqnarray*}
Furthermore, we can bound the conditional entropy of $L$ given $Z$ as
$0 \leq H(L|Z) \leq H(L)$.  This gives upper and lower bounds on the
entropy of $Z$ as
\begin{eqnarray*}
   H(Z|L) &\leq& H(Z)  \;\;\leq\;\;  H(Z|L) + H(L).
\end{eqnarray*}

\section{Proof of Lemma~\ref{LemExsBin}}
\label{PfLemExsBin}

We first derive upper and lower bounds in the case when $\spgam\kdim
\leq 1$.  We can rewrite the binomial distribution as
\begin{eqnarray*}
   p(\ell) &\defn& {\kdim \choose \ell} \spgam^\ell (1-\spgam)^{\kdim-\ell} 
   \;\; = \;\; \frac{\spgam\kdim}{\ell} \; {\kdim-1 \choose \ell-1}
   \spgam^{\ell-1} (1-\spgam)^{\kdim-\ell}
\end{eqnarray*}
and hence
\begin{eqnarray*}
   E &=& \frac{1}{2} \sum_{\ell=1}^\kdim \log\left(1 +
   \frac{\ell\minval^2}{\spgam}\right) p(\ell) \\
   &=& \frac{1}{2} \spgam\kdim \sum_{\ell=1}^\kdim \frac{\log\left(1 +
   \frac{\ell\minval^2}{\spgam}\right)}{\ell} \; {\kdim-1 \choose
   \ell-1} \spgam^{\ell-1} (1-\spgam)^{\kdim-\ell}.
\end{eqnarray*}
Taking the first two terms of the binomial expansion of $\left(1 +
\frac{\minval^2}{\spgam}\right)^\ell$ and noting that all the terms
are non-negative, we obtain the inequality
\begin{eqnarray*}
   \left(1 + \frac{\minval^2}{\spgam}\right)^\ell &\geq& 1 +
   \frac{\ell\minval^2}{\spgam}
\end{eqnarray*}
and consequently $\log\left(1 + \frac{\minval^2}{\spgam}\right) \geq
\frac{1}{\ell} \log\left(1 + \frac{\ell\minval^2}{\spgam}\right)$.
Using a change of variables (by setting $\ell'=\ell-1$) and applying
the binomial theorem, we thus obtain the upper bound
\begin{eqnarray*} 
   E &\leq& \frac{1}{2} \spgam\kdim \sum_{\ell=1}^\kdim \log\left(1 +
   \frac{\minval^2}{\spgam}\right) {\kdim-1 \choose \ell-1}
   \spgam^{\ell-1} (1-\spgam)^{\kdim-\ell} \\
   &=& \frac{1}{2} \spgam\kdim \log\left(1 +
   \frac{\minval^2}{\spgam}\right) \sum_{\ell'=0}^{\kdim-1} {\kdim-1
   \choose \ell'} \spgam^{\ell'} (1-\spgam)^{\kdim-\ell'-1} \\
  &=& \frac{1}{2} \spgam\kdim \log\left(1 +
  \frac{\minval^2}{\spgam}\right).
\end{eqnarray*}

To derive the lower bound, we will use the fact that $1+x \leq \e^x$
for all $x\in\R$, and $\e^{-x} \leq 1 - \frac{x}{2}$ for $x \in
[0,1]$.  
\begin{eqnarray*}
   E &=& \frac{1}{2} \sum_{\ell=1}^\kdim \log \left(1 +
   \frac{\ell\minval^2}{\spgam}\right) p(\ell) \\
   &\geq& \frac{1}{2} \log\left(1 + \frac{\minval^2}{\spgam}\right)
   \sum_{\ell=1}^\kdim p(\ell) \\
   &=& \frac{1}{2} \log\left(1 + \frac{\minval^2}{\spgam}\right) (1 -
   (1-\spgam)^\kdim) \\
   &\geq& \frac{1}{2} \log\left(1 + \frac{\minval^2}{\spgam}\right) (1
   - \e^{-\spgam\kdim}) \\
   &\stackrel{(a)}{\geq}& \frac{1}{2} \log\left(1 +
   \frac{\minval^2}{\spgam}\right) \left(\frac{\spgam\kdim}{2}\right).
\end{eqnarray*}  

Next, we examine the case when $\spgam\kdim = \tau$ for some constant
$\tau$.  The derivation of the upper bound in the case when
$\spgam\kdim \leq 1$ holds for the $\spgam\kdim = \tau$ case as well.
The proof of the lower bound follows the same steps as in the
$\spgam\kdim \leq 1$ case, except that we stop before applying the
last inequality~$(a)$.

Finally, we derive bounds in the case when $\spgam\kdim > 3$.  Since
the mean of a $L \sim \Bin(\kdim,\spgam)$ random variable is
$\spgam\kdim$, by Jensen's inequality the following bound always
holds,
\begin{eqnarray*}
   \Exs_L\left[\frac{1}{2} \log\left(1 + \frac{L\minval^2}{\spgam}
   \right)\right] &\leq& \frac{1}{2} \log(1 + \kdim\minval^2).
\end{eqnarray*}
To derive a matching lower bound, we use the fact that the median of a
$\Bin(\kdim,\spgam)$ distribution is one of
$\{\lfloor\spgam\kdim\rfloor - 1, \lfloor\spgam\kdim\rfloor,
\lfloor\spgam\kdim\rfloor + 1\}$.
This allows us to bound
\begin{eqnarray*}
   E &\geq& \frac{1}{2} \sum_{\ell=\lfloor\spgam\kdim\rfloor-1}^\kdim
   \log\left(1 + \frac{\ell\minval^2}{\spgam}\right) p(\ell) \\
   &\geq& \frac{1}{2} \log\left(1 + \frac{(\lfloor\spgam\kdim\rfloor -
   1) \minval^2}{\spgam}\right) \sum_{\ell=\lfloor\spgam\kdim\rfloor -
   1}^\kdim p(\ell) \\
  &\geq& \frac{1}{4} \log\left(1 + \frac{\kdim\minval^2}{3}\right)
\end{eqnarray*}
where in the last step we used the fact that $\frac{(\lfloor \spgam
\kdim \rfloor - 1) \minval^2}{\spgam} \geq \frac{(\spgam \kdim - 2)
\minval^2}{\spgam} \geq \frac{\kdim\minval^2}{3}$ for $\spgam\kdim
> 3$, and $\sum_{\ell=\text{median}}^\kdim p(\ell) \geq
\frac{1}{2}$.

\section{Bounds on binomial entropy}
\label{AppBinEntropy}

\begin{lemma}
Let $L \sim \Bin(\kdim,\spgam)$.  Then 
\begin{eqnarray*}
H(L) &\leq& \kdim \Hbin(\spgam).
\end{eqnarray*}  
Furthermore, if $\spgam = o\left(\frac{1}{\kdim\log{\kdim}}\right)$,
then $\kdim \Hbin(\spgam) \rightarrow 0$ as $\kdim \rightarrow
\infty$.
\end{lemma}

\begin{proof}
We can express the binomial variate as $L=\sum_{i=1}^\kdim Z_i$, where
$Z_i \sim$ \Ber($\spgam$) i.i.d.  Since $H(g(Z_1,\ldots,Z_\kdim))
\leq H(Z_1,\ldots,Z_\kdim)$, we have
\begin{eqnarray*}
  H(L) &\leq& H(Z_1,\ldots,Z_\kdim) \;\;=\;\; \kdim \Hbin(\spgam).
\end{eqnarray*}

Next we find the limit of $\kdim \Hbin(\spgam) = \kdim\spgam
\log\frac{1}{\spgam} + \kdim (1-\spgam) \log\frac{1}{1-\spgam}$.  Let
$\spgam = \frac{1}{\kdim f(\kdim)}$, and assume that $f(\kdim)
\rightarrow \infty$ as $\kdim \rightarrow \infty$.  Hence the first
term can be written as
\begin{eqnarray*}
   \kdim\spgam \log\frac{1}{\spgam} &=& \frac{1}{f(\kdim)} \log(\kdim
   f(\kdim)) \;\;=\;\; \frac{\log \kdim}{f(\kdim)} + \frac{\log
   f(\kdim)}{f(\kdim)},
\end{eqnarray*}
and so $\kdim\spgam \log\frac{1}{\spgam} \rightarrow 0$ if
$f(k)=\omega(\log k)$.  The second term can also be expanded as
\begin{eqnarray*}
   -\kdim (1-\spgam) \log(1-\spgam) &=& - \kdim \log\left(1 -
   \frac{1}{\kdim f(\kdim)}\right) + \frac{1}{f(\kdim)} \log\left(1 -
   \frac{1}{\kdim f(\kdim)}\right) \\
   &=& - \log \left(1 - \frac{1}{\kdim f(\kdim)}\right)^\kdim +
   \frac{1}{f(\kdim)} \log \left(1 - \frac{1}{\kdim f(\kdim)}\right).
\end{eqnarray*} 
If $f(\kdim) \rightarrow \infty$ as $\kdim \rightarrow \infty$, then
we have the limits
\begin{equation*}
   \lim_{\kdim\rightarrow\infty} \left(1 - \frac{1}{\kdim
   f(\kdim)}\right)^\kdim \; = \; 1 \qquad \text{and} \qquad
   \lim_{\kdim\rightarrow\infty} \left(1 - \frac{1}{\kdim
   f(\kdim)}\right) \; = \; 1,
\end{equation*}
which in turn imply that
\begin{equation*}
   \lim_{\kdim\rightarrow\infty} \log \left(1 - \frac{1}{\kdim
   f(\kdim)}\right)^\kdim \; = \; 0 \qquad \text{and} \qquad
   \lim_{\kdim\rightarrow\infty} \frac{1}{f(\kdim)} \log \left(1 -
   \frac{1}{\kdim f(\kdim)}\right) \; = \; 0.
\end{equation*}
\end{proof}

\begin{lemma}
Let $L \sim \Bin(\kdim,\spgam)$, then 
\begin{eqnarray*}
  H(L) &\leq& \frac{1}{2}\log(2\pi\e (k\gamma(1-\gamma) + \frac{1}{12})).
\end{eqnarray*}
\end{lemma}

\begin{proof}
We immediately obtain this bound by applying the differential entropy
bound on discrete entropy~\cite{Cover}.
\end{proof}

%%%%%%%%%%%%%%%%%%%%%%%%%%%%%%%%%%%%%%%%%%%%%%%%%%%%%%%%%%%%%%%%%%%%%%%%
\bibliographystyle{plain} \bibliography{mjwain_super}
%%%%%%%%%%%%%%%%%%%%%%%%%%%%%%%%%%%%%%%%%%%%%%%%%%%%%%%%%%%%%%%%%%%%%%%
\end{document}